\documentclass[a4paper,11pt]{amsproc}

\usepackage{amssymb}
\usepackage{mathtools}
\usepackage{hyperref}

\newtheorem{theorem}{Theorem}

\newtheorem{corollary}[theorem]{Corollary}
\newtheorem{prop}[theorem]{Proposition}
\theoremstyle{remark}
\newtheorem{remark}[theorem]{Remark}

\newcommand{\norm}[1]{\left\lVert #1\right\rVert}
\newcommand{\cE}{\mathcal E}
\newcommand{\cH}{\mathcal H}

\title[Sublinear growth and uniform convexity]{Sublinear growth of $1$-cocycles and \\ uniform convexity}
\author{Andreas Thom}
\address{A.T., Institute of Geometry, TU Dresden, 01062 Dresden, Germany}
\email{andreas.thom@tu-dresden.de}

\date{\today}

\subjclass[2010]{20F65, 22D12, 46B20}
\keywords{1-cocycle, uniformly convex Banach space, affine isometric action, $c_0$-representation, word length}

\begin{document}

\begin{abstract}

Let G be a finitely generated group, let $\pi \colon G \to {\rm GL}(\cE)$ be a uniformly bounded $c_0$-representation on a superreflexive Banach space $\cE$, and let $b \colon G \to \cE$ be a $1$-cocycle for $\pi$. Then $b$ has sublinear growth with respect to the word length. As a corollary we obtain the corresponding Hilbert space statement for strongly mixing unitary representations.
\end{abstract}

\maketitle

\section{Introduction}

Let $G$ be a finitely generated group with a fixed finite symmetric generating set, and let $|g|$ denote the associated word length. Let $\cE$ be a Banach space. Let
\[
\pi\colon G\to {\rm GL}(\cE)
\]
be a representation by linear isomorphisms. The representation $\pi$ is called \emph{uniformly bounded} if there exists $C>0$ such that $\norm{\pi(g)}\le C$ for every $g\in G$. We say that $\pi$ is a \emph{$c_{0}$-representation} if for every $\xi\in \cE$ and every $\varphi\in \cE^{\ast}$ one has
\[
\varphi(\pi(g)\xi)\to 0
\qquad\text{as } g\to\infty.
\]
Equivalently, every orbit vector tends weakly to $0$ at infinity; this terminology is standard in~\cite{BFGM07,CTVBanach08}.
An \emph{affine isometric action} with linear part $\pi$ is encoded by a $1$-cocycle
\[
b(gh)=b(g)+\pi(g)b(h)
\qquad (g,h\in G).
\]
The \emph{modulus of convexity} is defined for a Banach space $\cE$ by
\[
\delta_{\cE}(\varepsilon):=\inf\Bigl\{1-\Bigl\lVert\frac{x+y}{2}\Bigr\rVert: \norm{x}=\norm{y}=1,\ \norm{x-y}\ge \varepsilon\Bigr\}
\qquad (0<\varepsilon\le 2).
\]
A Banach space is called \emph{uniformly convex} if $\delta_{\cE}(\varepsilon)>0$ for every $\varepsilon>0$; see, for example, Benyamini--Lindenstrauss~\cite{BL00}. A Banach space $\cE$ is called \emph{superreflexive} if it admits an equivalent uniformly convex norm; see, for example, \cite{BFGM07,BL00} for background.

The purpose of this note is to show that for uniformly bounded $c_{0}$-representations on superreflexive Banach spaces, every $1$-cocycle has sublinear growth along the word length, see Theorem \ref{thm:main}. 

In the Hilbert space setting, the corresponding class of unitary (or more generally, uniformly bounded) representations are called \emph{strongly mixing}. Growth and rigidity properties of $1$-cocycles into unitary representations are intertwined with affine isometric actions, reduced cohomology, and harmonic cocycles; classical references include Guichardet~\cite{Guichardet72}, Shalom~\cite{Shalom00}, and the monograph of Bekka--de la Harpe--Valette~\cite{BdlHV08}. More recently, Cornulier--Tessera--Valette~\cite{CTV07} analyzed the relation between sublinear growth and almost coboundaries, while Chifan--Sinclair~\cite{CS15} and Ozawa~\cite{Oz18} obtained vanishing and weak sublinearity results for cocycles into weakly mixing representations of amenable groups; see also Gournay~\cite{Gou18} for related mixing/cohomological phenomena.

To the best of our knowledge, even the Hilbert space statement for strongly mixing representations (stated here as Corollary~\ref{cor:hilbert}) does not seem to be explicitly recorded in the literature, despite the fact that quantitative and qualitative control of cocycles remains an active topic of current research. 

\medskip

Our main result is the following theorem. It will be proved in the next section.
\begin{theorem}\label{thm:main}
Let $G$ be a finitely generated group, $\cE$ be a superreflexive Banach space,
$
\pi\colon G\to {\rm GL}(\cE)
$
be a uniformly bounded $c_{0}$-representation, and  $b\colon G\to \cE$ be a $1$-cocycle. Then
\[
\limsup_{|g|\to\infty}\frac{\norm{b(g)}}{|g|}=0.
\]
Equivalently, every such $1$-cocycle has sublinear growth.
\end{theorem}

Applications to the setting of Hilbert spaces are discussed in the third section, where we also show that the weaker notion of weak mixing does not force sublinear growth of $1$-cocycles, by constructing a weakly mixing representation of the free group $F_2$ together with a cocycle of linear growth along a sequence, see Proposition \ref{prop:weak}. 

A final section records a consequence of Theorem~\ref{thm:main} for vanishing of reduced cohomology for finitely generated amenable groups admitting controlled Følner sequences, see Corollary~\ref{cor:reducedcohozero}. This extends results of Tessera~\cite{Tess09}, who proved the vanishing of reduced cohomology for the $L^p$-representation associated with a mixing measure preserving actions $G \curvearrowright (X,\mu)$, where $\mu$ is an infinite measure and $1< p < \infty$.

\section{Proof of the main theorem}

\begin{proof}[Proof of Theorem~\ref{thm:main}]
At first, we use \cite[Proposition~2.3]{BFGM07} to replace the given norm on $\cE$ by an equivalent uniformly convex norm that is preserved by $\pi$. Hence, we may assume that $\cE$ is uniformly convex and $\pi$ is an isometric representation. Passing to the underlying real Banach space, we may assume that $\cE$ is real.
Assume for contradiction that
\[
\alpha:=\limsup_{|g|\to\infty}\frac{\norm{b(g)}}{|g|}>0.
\]
For $g\ne e$ set
\[
q(g):=\frac{\norm{b(g)}}{|g|}.
\]

\medskip

\noindent
\emph{Step 1.}
Fix
$
\Delta:=\delta_{\cE}(1/2)>0
$
and choose $\eta>0$ so small that
\[
0<\eta<\frac{\alpha\Delta}{8}.
\]
By the definition of $\limsup$, there exists $R\in \mathbb N$ such that $|x|\ge R$ implies $q(x)\le \alpha+\eta.$
Equivalently,
\[
|x|\ge R \implies \norm{b(x)}\le (\alpha+\eta)|x|.
\tag{1}
\]

\medskip

\noindent
\emph{Step 2.}
For $n\ge 1$, let
$
\alpha_{n}:=\max\{\norm{b(g)}: |g|=n\}.
$
Since spheres of radius $n$ are finite, the maximum exists. The sequence $(\alpha_{n})_{n\ge 1}$ is subadditive: if $|g|=n+m$ and $g=xy$ is obtained by cutting a geodesic word for $g$ after $n$ letters, then $|x|=n$, $|y|=m$, and
\[
\norm{b(g)}\le \norm{b(x)}+\norm{b(y)}\le \alpha_{n}+\alpha_{m}.
\]
Hence Fekete's lemma gives the existence of
\[
\lim_{n\to\infty}\frac{\alpha_{n}}{n}=\inf_{n\ge 1}\frac{\alpha_{n}}{n}.
\]
Since $\alpha_{n}/n=\max_{|g|=n} q(g)$, this limit equals $\alpha$. In particular, for every $n$ there exists $g_{n}\in G$ such that
\[
|g_{n}|=n
\qquad\text{and}\qquad
\norm{b(g_{n})}=\alpha_{n}\ge \alpha n.
\tag{2}
\]
Fix a geodesic word
$
g_{n}=s_{1}^{(n)}\cdots s_{n}^{(n)}
$
for each $n$.

\medskip

\noindent
\emph{Step 3.}
For $n$ large enough, partition the geodesic word for $g_{n}$ into consecutive geodesic subwords
\[
g_{n}=h_{1}^{(n)}h_{2}^{(n)}\cdots h_{m_{n}}^{(n)}
\]
such that for every $i$,
\[
R\le \ell_{i}^{(n)}:=|h_{i}^{(n)}|\le 2R,
\]
and
\[
\sum_{i=1}^{m_{n}}\ell_{i}^{(n)}=n.
\tag{3}
\]
Fix $n$ and suppress the superscript $(n)$ from the notation. For $1\le i\le m_{n}$ define the prefix
$
t_{i}:=h_{1}\cdots h_{i-1},
$
with $t_{1}=e$, and set
$
v_{i}:=\pi(t_{i})b(h_{i}).
$
Repeatedly applying the cocycle identity gives
\[
b(g_{n})=\sum_{i=1}^{m_{n}} v_{i}.
\tag{4}
\]
Since $\pi(t_{i})$ is an isometry,
$
\norm{v_{i}}=\norm{b(h_{i})}.
$
As $|h_{i}|\ge R$, estimate (1) yields
\[
\norm{v_{i}}\le (\alpha+\eta)\ell_{i}.
\tag{5}
\]
Because of (2), $b(g_{n})\ne 0$. Choose a norming functional $j_{n}\in \cE^{\ast}$ with
$
\norm{j_{n}}=1$ and $j_{n}(b(g_{n}))=\norm{b(g_{n})}.$

\medskip

\noindent
\emph{Step 4.}
For each block define
\[
d_{i}:=(\alpha+\eta)\ell_{i}-j_{n}(v_{i}).
\]
Since $j_{n}(v_{i})\le \norm{v_{i}}\le (\alpha+\eta)\ell_{i}$ by (5), we have $d_{i}\ge 0$. Summing and using (2), (3), and (4), we obtain
\begin{align*}
\sum_{i=1}^{m_{n}} d_{i}
&=(\alpha+\eta)\sum_{i=1}^{m_{n}}\ell_{i}-\sum_{i=1}^{m_{n}} j_{n}(v_{i})\\
&=(\alpha+\eta)n-j_{n}\Bigl(\sum_{i=1}^{m_{n}} v_{i}\Bigr)\\
&=(\alpha+\eta)n-j_{n}(b(g_{n}))\\
&=(\alpha+\eta)n-\norm{b(g_{n})}\\
&\le \eta n.
\end{align*}
Thus
\[
\sum_{i=1}^{m_{n}} d_{i}\le \eta n.
\]

Call an index $i$ \emph{good} if
$
d_{i}\le 4\eta\,\ell_{i}.
$
Let $G_{n}$ be the set of good indices. Then
\[
\sum_{i\notin G_{n}} 4\eta\,\ell_{i}<\sum_{i\notin G_{n}} d_{i}\le \sum_{i=1}^{m_{n}} d_{i}\le \eta n,
\]
so
\[
\sum_{i\notin G_{n}} \ell_{i}\le \frac14 n.
\]
Therefore
\[
\sum_{i\in G_{n}} \ell_{i}\ge \frac34 n.
\tag{6}
\]

\medskip

\noindent
\emph{Step 5.}
Fix $i\in G_{n}$. Since $d_{i}\le 4\eta\ell_{i}$, we have
\[
j_{n}(v_{i})\ge (\alpha-3\eta)\ell_{i}.
\tag{7}
\]
Combining (5) and (7), we get
\[
\norm{v_{i}}\ge j_{n}(v_{i})\ge (\alpha-3\eta)\ell_{i}>0.
\tag{8}
\]
Hence the normalized vector
\[
w_{i}:=\frac{v_{i}}{\norm{v_{i}}}
\in S_{\cE}
\]
is well defined, and
\[
j_{n}(w_{i})=\frac{j_{n}(v_{i})}{\norm{v_{i}}}
\ge \frac{(\alpha-3\eta)\ell_{i}}{(\alpha+\eta)\ell_{i}}
=1-\frac{4\eta}{\alpha+\eta}
>1-\Delta.
\]
Now fix $i,j\in G_{n}$. Since $\norm{w_{i}}=\norm{w_{j}}=1$ and $\norm{j_{n}}=1$, we get
\[
\Bigl\lVert\frac{w_{i}+w_{j}}{2}\Bigr\rVert
\ge j_{n}\Bigl(\frac{w_{i}+w_{j}}{2}\Bigr)
=\frac{j_{n}(w_{i})+j_{n}(w_{j})}{2}
>1-\Delta.
\]
By the definition of $\Delta=\delta_{\cE}(1/2)$, this forces $\norm{w_{i}-w_{j}}<\frac12$, hence
\[
\norm{w_{i}-w_{j}}<\frac12.
\tag{9}
\]

\medskip

\noindent
\emph{Step 6.}
Let
$
\mathcal A_{R}:=\{h\in G: R\le |h|\le 2R\}.
$
This is a finite set; write $M_{R}:=|\mathcal A_{R}|$. Since each good block has length at most $2R$, estimate (6) implies
\[
2R\,|G_{n}|\ge \sum_{i\in G_{n}}\ell_{i}\ge \frac34 n,
\]
hence
\[
|G_{n}|\ge \frac{3}{8R}n.
\]
Among the good blocks, some element of $\mathcal A_{R}$ occurs at least $|G_{n}|/M_{R}$ times. Therefore there exist $h^{(n)}\in \mathcal A_{R}$ and a subset $I_{n}\subset G_{n}$ such that
\[
h_{i}=h^{(n)} \quad (i\in I_{n}),
\]
and
\[
|I_{n}|\ge \frac{3}{8RM_{R}}n.
\]
Passing to a subsequence, we may assume that $h^{(n)}=h$ is independent of $n$. Set
$
X_{n}:=\{t_{i}: i\in I_{n}\}.
$
Since the prefixes $t_{i}$ are distinct, we have
\[
|X_{n}|=|I_{n}|\ge \frac{3}{8RM_{R}}n,
\]
so $|X_{n}|\to\infty$.

For $t\in X_{n}$, the corresponding vector is
$
\pi(t)b(h)=v_{i}
$
for the unique $i\in I_{n}$ with $t=t_{i}$. By (8),
\[
\norm{b(h)}=\norm{v_{i}}\ge (\alpha-3\eta)|h|>0.
\tag{10}
\]
Moreover, if $t,t'\in X_{n}$ correspond to good indices $i,j\in I_{n}$, then (9) gives
\[
\lVert\pi(t)b(h)-\pi(t')b(h)\rVert<\frac12\norm{b(h)}.
\tag{11}
\]

\medskip

\noindent
\emph{Step 7.}
Since $|X_{n}|\to\infty$ and balls in $G$ are finite, the difference sets $X_{n}^{-1}X_{n}$ cannot stay inside one fixed finite subset of $G$. Thus, we may choose
\[
\gamma_{n}\in X_{n}^{-1}X_{n}
\qquad\text{with}\qquad
|\gamma_{n}|\to\infty.
\]
Write $\gamma_{n}=t_{n}^{-1}s_{n}$ with $s_{n},t_{n}\in X_{n}$. Then by (11),
\[
\norm{\pi(\gamma_{n})b(h)-b(h)}
=\norm{\pi(t_{n})^{-1}(\pi(s_{n})b(h)-\pi(t_{n})b(h))}<\frac12\norm{b(h)}.
\tag{12}
\]
Choose a norming functional $\psi\in \cE^{\ast}$ for $b(h)$, so $\norm{\psi}=1$ and $\psi(b(h))=\norm{b(h)}$. By (12),
\[
\psi(\pi(\gamma_{n})b(h))
\ge \psi(b(h))-\norm{\pi(\gamma_{n})b(h)-b(h)}
>\frac12\norm{b(h)}
\qquad (n\ge 1).
\tag{13}
\]
On the other hand, since $\pi$ is a $c_{0}$-representation and $|\gamma_{n}|\to\infty$, we must have
$
\psi(\pi(\gamma_{n})b(h))\to 0,
$
contradicting (13). This contradiction proves the theorem.
\end{proof}

\section{The setting of Hilbert spaces}

\begin{corollary}\label{cor:hilbert}
Let $G$ be a finitely generated group, let $\pi\colon G\to \mathcal U(\cH)$ be a strongly mixing unitary representation on a Hilbert space $\cH$, and let $b\colon G\to \cH$ be a $1$-cocycle. Then
\[
\limsup_{|g|\to\infty}\frac{\norm{b(g)}}{|g|}=0.
\]
\end{corollary}

\begin{proof}
Every Hilbert space is uniformly convex. Moreover, for unitary representations on Hilbert space, the condition of being $c_{0}$ is exactly the usual strong mixing condition, since the dual can be identified with the Hilbert space itself through the inner product. Therefore Corollary~\ref{cor:hilbert} is an immediate consequence of Theorem~\ref{thm:main}.
\end{proof}

A weaker notion of mixing is that of weak mixing, which is equivalent to the absence of non-zero finite-dimensional subrepresentations. In this section we show that weak mixing does not force sublinear growth of $1$-cocycle; a fact that is surely known to experts. More precisely, we construct a weakly mixing representation of the free group $F_2$ together with a cocycle of linear growth along a sequence.

\begin{prop}\label{prop:weak}
There exists a finitely generated group $G$, a weakly mixing unitary representation $\pi\colon G\to \mathcal U(\cH)$, and a $1$-cocycle $b\colon G\to \cH$ such that
\[
\limsup_{|g|\to\infty}\frac{\|b(g)\|}{|g|}=1.
\]
\end{prop}
\begin{proof}

Let $F_2=\langle x,y\rangle$, and identify $\mathbb Z$ with the subgroup generated by $x$. We claim that the quasi-regular representation
\[
\lambda_{F_2/\mathbb Z}\colon F_2\to \mathcal U(\ell^2(F_2/\mathbb Z))
\]
is weakly mixing, and there exists a $1$-cocycle
$
b\colon F_2\to \ell^2(F_2/\mathbb Z)
$
such that
$
b(x^n)=n\,\delta_{\mathbb Z}
$ for all $n\ge 1$.
Since $\mathbb Z$ has infinite index in $F_2$, the representation $\lambda_{F_2/\mathbb Z}$ is weakly mixing. One way to see this is that weak mixing of $\lambda_{F_2/\mathbb Z}$ is equivalent to the absence of non-trivial invariant vectors in $\ell^2(F_2/\mathbb Z \times F_2/\mathbb Z)$ with respect to the diagonal $F_2$-action, and this is equivalent to the absence of finite $F_2$-orbits in
$
F_2/\mathbb Z\times F_2/\mathbb Z.
$

Since $F_2$ is the free group on the generators $x$ and $y$, a $1$-cocycle for a unitary representation of $F_2$ is uniquely determined by its values on $x$ and $y$, and these values can be prescribed arbitrarily. We apply this to the representation $\lambda_{F_2/\mathbb Z}$ and choose
\[
b(x)=\delta_{\mathbb Z},
\qquad
b(y)=0.
\]
Thus there exists a unique $1$-cocycle
$
b\colon F_2\to \ell^2(F_2/\mathbb Z)
$
with these values.

Since $x\in \mathbb Z$, the coset $\mathbb Z$ is fixed by left multiplication with $x$, so
$
\lambda_{F_2/\mathbb Z}(x)\delta_{\mathbb Z}=\delta_{\mathbb Z}.
$
Using the cocycle identity, it follows by induction that for every $n\ge 1$,
\[
b(x^n)=b(x^{n-1})+\lambda_{F_2/\mathbb Z}(x^{n-1})b(x)=b(x^{n-1})+\delta_{\mathbb Z}=n\,\delta_{\mathbb Z}.
\]
This proves the claim, and hence the proposition.
\end{proof}

\begin{remark}
It is an interesting open question to decide if an example as in Proposition~\ref{prop:weak} can be found for an amenable group, see \cite{CS15,CTV07,Oz18}.
\end{remark}

\section{Vanishing of reduced cohomology for amenable groups}

In this section we record a consequence of Theorem~\ref{thm:main} for finitely generated amenable groups admitting controlled Følner sequences. Here, a \emph{controlled Følner sequence} is a Følner sequence $(F_n)_{n\ge 1}$ such that there exists $C>0$ with 
$$|\partial F_n|\le \frac{C |F_n|}{\mathrm{diam}(F_n)}$$ for every $n\ge 1$, where $\mathrm{diam}(F_n)$ is the diameter of $F_n$ with respect to the word metric; see, for example, Tessera~\cite{Tess09}.

Cornulier and Tessera proved that for uniformly bounded Banach $G$-modules, sublinear 1-cocycles are almost coboundaries whenever the group admits a controlled Følner sequence; more precisely, this is Proposition~5.6 in \cite{CoTe20}. Therefore we obtain the following corollary:

\begin{corollary}\label{cor:reducedcohozero}
Let $G$ be a finitely generated amenable group admitting a controlled Følner sequence. Let $\pi$ be a uniformly bounded $c_0$-representation of $G$ on a superreflexive Banach space $\cE$. Then, the first reduced cohomology vanishes, $H^1_{\rm red}(G,\pi)=0.$ Equivalently, every 1-cocycle $b \colon G \to \cE$ is an almost coboundary.
\end{corollary}

Corollary~\ref{cor:reducedcohozero} appears to be new at this level of generality. Indeed, Tessera proved in \cite{Tess09} the vanishing of reduced first cohomology for mixing $L^p$-representations, $1<p<\infty$, for a class of amenable groups with controlled Følner sequences. Since $L^p$-spaces are uniformly convex for $1<p<\infty$, the above corollary extends this vanishing statement from mixing $L^p$-representations to arbitrary isometric $c_0$-representations on uniformly convex Banach spaces.

Let us recall some classes of amenable groups known to admit controlled Følner sequences. Tessera studied controlled Følner sequences in \cite{Tess09} and recorded that such exist for polycyclic groups and several solvable groups of exponential growth, including semidirect products of the form $\mathbb Z[1/mn]\rtimes_{m/n}\mathbb Z$ (hence in particular the solvable Baumslag--Solitar groups $\mathrm{BS}(1,m)$). Finally, groups of polynomial growth also admit controlled Følner sequences, for instance by taking word balls. 
+
\begin{remark}
We are not aware of an explicitly recorded example in the literature of a finitely generated amenable group that does not admit a controlled Følner sequence; see the discussion in \cite{CS15}. However, such examples follow immediately from the following observation: a finitely generated amenable group admitting a controlled Følner sequence has Følner function at most exponential along a subsequence. If $(F_n)$ is a controlled Følner sequence, then $|\partial F_n|/|F_n| \leq C/\operatorname{diam}(F_n)$, so after setting $k_n \asymp \operatorname{diam}(F_n)$ one obtains $\operatorname{Føl}(k_n)\leq |F_n| \leq \exp(C'k_n)$. In particular, groups with superexponential Følner function cannot admit controlled Følner sequences. By Erschler's theorem \cite[Theorem~1]{Erschler03}, this applies for instance to the wreath product $\mathbf Z \wr \mathbf Z$.
\end{remark}

\section*{Acknowledgements}

The author thanks Alain Valette for helpful comments on an earlier version of this note.
GPT-5.4 was used to assist in drafting parts of this manuscript. All content was reviewed and substantially revised by the author, who is responsible for the final text.

\end{document}